\documentclass[12pt, oneside, a4paper]{article}

\usepackage{amsmath}
\usepackage{amsfonts}
\usepackage{amssymb}
\usepackage{amsthm,mathrsfs}
\usepackage{enumerate}
\usepackage{graphicx}
\newtheorem{theorem}{Theorem}[section]
\newtheorem{lemma}[theorem]{Lemma}

\theoremstyle{definition}

\title{\textbf{$K_4$-free character graphs with seven vertices}}
\author{Mahdi Ebrahimi\footnote{ m.ebrahimi.math@ipm.ir}
 \\
 {\small\em  School of Mathematics, Institute for Research in Fundamental Sciences (IPM)},\\{\small\em P.O. Box: 19395--5746, Tehran, Iran}}

\date{}

\begin{document}

\maketitle

\begin{abstract}
For a finite group $G$, let $\Delta(G)$ denote the character graph built on the set of degrees of the irreducible complex characters of $G$. In this paper,
 we determine the structure of all
 finite groups $G$ with $K_4$-free character graph $\Delta(G)$ having seven vertices.
  We also obtain a classification of all $K_4$-free graphs with seven vertices which can occur as character graphs of some finite groups.

 \end{abstract}

\section{Introduction}
$\noindent$ Let $G$ be a finite group and $R(G)$ be the solvable radical of $G$. Also let ${\rm cd}(G)$ be the set of all character degrees of $G$, that is,
 ${\rm cd}(G)=\{\chi(1)|\;\chi \in {\rm Irr}(G)\} $, where ${\rm Irr}(G)$ is the set of all complex irreducible characters of $G$. The set of prime divisors of character degrees of $G$ is denoted by $\rho(G)$. It is well known that the
 character degree set ${\rm cd}(G)$ may be used to provide information on the structure of the group $G$. For example, Ito-Michler's Theorem \cite{[C]} states that if a prime $p$ divides no character degree of a finite group $G$, then $G$ has a
 normal abelian Sylow $p$-subgroup. Another result due to Thompson \cite{[D]} says that if a prime $p$ divides
 every non-linear character degree of a group $G$, then $G$ has a normal $p$-complement.

A useful way to study the character degree set of a finite group $G$ is to associate a graph to ${\rm cd}(G)$.
One of these graphs is the character graph $\Delta(G)$ of $G$ \cite{[I]}. Its vertex set is $\rho(G)$ and two vertices $p$ and $q$ are joined by an edge if the product $pq$ divides some character degree of $G$. We refer the readers to a survey by Lewis \cite{[M]} for results concerning this graph and related topics.

 If we know the structure of $\Delta(G)$, we can often say a lot about the structure of the group $G$. For instance, Casolo et al. \cite{[SDM]} has proved that if for a finite solvable group $G$,  $\Delta(G)$ is connected with diameter $3$, then there exists a prime $p$ such that $G=PH$, with $P$ a normal non-abelian Sylow $p$-subgroup of $G$ and $H$ a $p$-complement. For another instance, all finite solvable groups $G$ whose character graph $\Delta(G)$ is disconnected have been completely classified by Lewis \cite{los}. In \cite{[AT]}, it was shown that when the character graph $\Delta(G)$ of a finite group $G$ is $K_4$-free, then $\Delta(G)$ has at most $7$ vertices.
 In this paper, we wish to determine the structure of all finite groups $G$ with $K_4$-free character graph $\Delta(G)$ having seven vertices.

  The solvable group $G$ is said to be disconnected if $\Delta(G)$ is disconnected. Also $G$ is called of disconnected Type $n$ if $G$ satisfies the hypotheses of Example $2.n$ in \cite{los}. Now let $\Gamma$ be a finite graph. If $\Gamma_1,\Gamma_2,...,\Gamma_n$ are connected components of $\Gamma$, we use the notation $\Gamma=\Gamma_1\cup \Gamma_2\cup ...\cup \Gamma_n$ to determine the connected components of $\Gamma$. Also note that for an integer $n\geqslant 1$,  the set of prime divisors of $n$ is denoted by $\pi(n)$. Now we are ready to state the main result of this paper.\\

\noindent \textbf{Main Theorem.}  \textit{Let $G$ be a finite group and $\Delta(G)$ be a $K_4$-free graph with seven vertices. Then For some integer $f\geqslant 2$, $G\cong \rm{PSL}_2(2^f)\times R(G)$, where $|\pi(2^f\pm 1)|=1,2$ or $3$. Also\\
i) If $|\pi(2^f\pm 1)|=1$, then for some disconnected groups $A$ and $B$ of disconnected Types 1 or 4, $R(G)\cong A\times B$ and $\Delta(G)\cong K^c_3\star C_4$.\\
ii) If $|\pi(2^f\pm 1)|=2$, then  $R(G)$ is a disconnected group of disconnected Type 1 or 4 and $\Delta(G)\cong (K_2\cup K_1\cup K_2)\star K_2^c$.\\
iii) If $|\pi(2^f\pm 1)|=3$, then  $R(G)$ is abelian and $\Delta(G)\cong K_3\cup K_1\cup K_3$.}

\section{Preliminaries}
$\noindent$ In this paper, all groups are assumed to be finite and all
graphs are simple and finite. For a finite group $G$, the set of prime divisors of $|G|$ is denoted by $\pi(G)$.
If $H\leqslant G$ and $\theta \in \rm{Irr}(H)$, we denote by $\rm{Irr}(G|\theta)$ the set of irreducible characters of $G$ lying over $\theta$ and define $\rm{cd}(G|\theta):=\{\chi(1)|\,\chi \in \rm{Irr}(G|\theta)\}$. We frequently use, Clifford's Theorem which can be found as Theorem 6.11 of \cite{[isa]} and Gallagher's Theorem which is corollary 6.17 of \cite{[isa]}. Also if $N\lhd G$ and $\theta\in \rm{Irr}(N)$, the inertia subgroup of $\theta$ in $G$ is denoted by  $I_G(\theta)$. We begin with Corollary 11.29 of \cite{[isa]}.

\begin{lemma}\label{fraction}
Let $ N \lhd G$ and $\varphi \in \rm{Irr}(N)$. Then for every $\chi \in \rm{Irr}(G|\varphi)$, $\chi(1)/\varphi(1)$ divides $[G:N]$.
\end{lemma}
\begin{lemma}\label{good}\cite{[Ton]}
Let $N$ be a normal subgroup of a group $G$ so that $G/N\cong S$, where $S$ is a non-abelian simple group. Let $\theta \in \rm{Irr}(N)$. Then either $\chi (1)/\theta(1)$ is divisible by two distinct primes in $\pi(G/N)$ for some $\chi \in \rm{Irr}(G|\theta)$ or $\theta$ is extendible to $\theta_0\in \rm{Irr}(G)$ and $G/N\cong A_5$ or $\rm{PSL}_2(8)$.
\end{lemma}

\begin{lemma}\label{method}
Suppose $G$ is a finite group, $N\lhd G$, $\theta \in \rm{Irr}(N)$ is $G$-invariant and $P/N$ is a normal Sylow $p$-subgroup of $G/N$. If all Sylow subgroups of $G/P$ are cyclic, then either $\theta$ extends to $G$ or all character degrees in $\rm{cd}(G|\theta)$ are divisible by $p\theta(1)$.
\end{lemma}

\begin{proof}
If $\theta$ extends to $P$, then as all Sylow subgroups of $G/P$ are cyclic,
$\theta$ is extendible to $G$.
Thus we assume that $\theta$ does not extend to $P$.
 Hence using Lemma \ref{fraction}, all character degrees in $\rm{cd}(P|\theta)$ are divisible by $p\theta(1)$.
  Now let $\chi \in \rm{Irr}(G|\theta)$.
   There is $\varphi \in \rm{Irr}(P|\theta)$ such that $\chi \in \rm{Irr}(G|\varphi)$.
    Thus as $\varphi(1) |\; \chi(1)$ , $\chi(1)$ is divisible by $p\theta(1)$.
\end{proof}

Let $\Gamma$ be a graph with vertex set $V(\Gamma)$ and edge set
$E(\Gamma)$. The complement of $\Gamma$ and the induced subgraph of $\Gamma$ on $X\subseteq V(\Gamma)$
 are denoted by $\Gamma^c$ and  $\Gamma[X]$, respectively. If $E(\Gamma)=\emptyset$, $\Gamma$ is called an empty graph.
  Now let $\Delta$ be a graph with vertex set $V(\Delta)$ for which $V(\Delta)\cap V(\Gamma)= \emptyset$. The join $\Gamma \ast \Delta$ of graphs $\Gamma$
and $\Delta$ is the graph $\Gamma \cup \Delta$ together with all
edges joining $V (\Gamma)$ and $V (\Delta)$. We use the notations $K_n$ for a complete graph with $n$ vertices and $C_n$ for a cycle of length $n$. If for some integer $n\geqslant 2$, $\Gamma$ does not contain a copy  of $K_n$ as an induced subgraph, then $\Gamma$ is called a $K_n$-free graph. We now state some relevant results on character graphs
needed in the next sections.

\begin{lemma}\label{pal}
(Palfy's condition) \cite {palfy} Let $G$ be a group and let $\pi \subseteq \rho(G)$.  If $G$ is solvable and $|\pi|\geqslant 3$, then there exist two distinct primes $u, v$ in $\pi$ and $\chi \in \rm{Irr}(G)$ such that $uv | \chi(1)$.
\end{lemma}

\begin{lemma}\label{square2}\cite{Sq}
Let $G$ be a solvable group with $\Delta(G)\cong C_4$. Then $G\cong A\times B$, where  $A$ and $B$ are disconnected groups. Also for some distinct primes $p,q,r$ and $s$, $\rho(A)=\{p,q\}$ and $\rho(B)=\{r,s\}$.
\end{lemma}

\begin{lemma}\label{square}\cite{Sq}
Let $G$ be a solvable group. If $\Delta(G)$ has at least $4$ vertices, then either $\Delta(G)$ contains a triangle or $\Delta(G)\cong C_4$.
\end{lemma}

The structure of the character graph of $\rm{PSL}_2(q)$ is determined as follows:

\begin{lemma}\label{chpsl}\cite{[white]}
Let $G\cong \rm{PSL}_2(q)$, where $q\geqslant 4$ is a power of a prime $p$.\\
\textbf{a)}
 If $q$ is even, then $\Delta(G)$ has three connected components, $\{2\}$, $\pi(q-1)$ and $\pi(q+1)$, and each component is a complete graph.\\
\textbf{b)}
 If $q>5$ is odd, then $\Delta(G)$ has two connected components, $\{p\}$ and $\pi((q-1)(q+1))$.\\
i)
 The connected component $\pi((q-1)(q+1))$ is a complete graph if and only if $q-1$ or $q+1$ is a power of $2$.\\
ii)
  If neither of $q-1$ or $q+1$  is a power of $2$, then $\pi((q-1)(q+1))$ can be partitioned as $\{2\}\cup M \cup P$, where $M=\pi (q-1)-\{2\}$ and $P=\pi(q+1)-\{2\}$ are both non-empty sets. The subgraph of $\Delta(G)$ corresponding to each of the subsets $M$, $P$ is complete, all primes are adjacent to $2$, and no prime in $M$ is adjacent to any prime in $P$.
 \end{lemma}

When $\Delta(G)^c$ is not a bipartite graph, then there exists a useful restriction on the structure of $G$ as follows:

\begin{lemma}\label{cycle} \cite{AC}
Let $G$ be a finite group and $\pi$ be a subset of the vertex set of $\Delta(G)$ such that $|\pi|> 1$ is an odd number. Then $\pi$ is the set of vertices of a cycle in $\Delta(G)^c$ if and only if $O^{\pi^\prime}(G)=S\times A$, where $A$ is abelian, $S\cong \rm{SL}_2(u^\alpha)$ or $S\cong \rm{PSL}_2(u^\alpha)$ for a prime $u\in \pi$ and a positive integer $\alpha$, and the primes in $\pi - \{u\}$ are alternately odd divisors of $u^\alpha+1$ and  $u^\alpha-1$.
\end{lemma}

\begin{lemma}\label{her} \cite{[her]}
Let $G$ be a simple group. If $|\pi(G)|=3$, then $G$ is isomorphic to one of the groups $A_5,A_6,\rm{PSL}_2(7),\rm{PSL}_2(8),\rm{PSL}_2(17),\rm{PSL}_3(3),\rm{PSU}_3(3)$ and $\rm{PSU}_4(2)$.
\end{lemma}

 \begin{lemma}\label{type}
  Let $G$ be a disconnected group. If $|\rho(G)|=2$ and $2 \notin \rho(G)$, then $G$ is of disconnected Type $1$ or $4$.
  \end{lemma}
  \begin{proof}
  Using descriptions of the six types of disconnected groups \cite{los}, we are done.
  \end{proof}

\section{About $\rm{PSL}_2(q)$}
$\noindent$ In order to the proof of Main Theorem, we will need some properties of $\rm{PSL}_2(q)$, where $q$ is a prime power.
We will make use of Dickson's list of the subgroups of $\rm{PSL}_2(q)$, which can be found as Hauptsatz II.8.27 of \cite{hup}.
 We also use the fact that the Schur multiplier of  $\rm{PSL}_2(q)$ is trivial unless $q=4$ or $q$ is odd, in which case it is of order $2$ if $q\neq 9$ and of order $6$ if $q=9$. We now state some relevant results on $\Delta(\rm{PSL}_2(q=2^f))$
needed in the next sections.
 We will use Zsigmondy's Theorem which can be found in \cite{zsi}.

 \begin{lemma} \label{interest}\cite{[e]}
 If $S=\rm{PSL}_2(2^f)$ and $|\rho(S)|=4$, then either:\\
 \textbf{a)} $f=4$, $2^f+1=17$ and $2^f-1=3.5$ or\\
 \textbf{b)} $f\geqslant 5$ is prime, $2^f-1=r$ is prime, and $2^f+1=3.t^\beta$, with $t$ an odd prime and $\beta\geqslant 1$ odd.
 \end{lemma}

 \begin{lemma}\label{evenfive}
 If $S=\rm{PSL}_2(q=2^f)$ and $|\pi(q\pm 1)|=2$, then either:\\
 \textbf{a)} $f$ is a prime.\\
 \textbf{b)} $f=6$ or $9$.
 \end{lemma}

 \begin{proof}
 Let $|\pi(f)|\geqslant 2$. Then for some distinct primes $p_1$ and $p_2$, $p_1p_2|f$. Thus one of the following cases occurs: \\
 a) $p_1p_2=6$. If $f=p_1p_2$, then we are done. Thus we suppose $f\neq p_1p_2$. Hence there is a prime $p$ so that $p_1p_2p\mid f$. Therefore by Zsigmondy's Theorem, there is a prime $s$ such that $s|2^{6p}-1$ and $s\nmid 2^6-1$. Thus $3.7.s\mid 2^f-1$ which is a contradiction. \\
 b) $p_1p_2\neq 6$. Let $p_1<p_2$. By Zsigmondy's Theorem, there are primes $s_1,s_2$ and $s_3$ so that $s_1\mid 2^{p_1p_2}-1$, $s_1\nmid 2^{p_1}-1$,  $s_1\nmid 2^{p_2}-1$,  $s_2\mid 2^{p_2}-1$,  $s_2\nmid 2^{p_1}-1$ and  $s_3\mid 2^{p_1}-1$. Hence $s_1,s_2$ and $s_3$ are distinct and $s_1s_2s_3\mid 2^f-1$. It is a contradiction. \\
 Therefore for some prime $p$ and positive integer $n$, $f=p^n$. If $f=9$, we are done.  Now let $n\geqslant 2$ and $f\neq 9$. If $p$ is odd, then by Zsigmondy's Theorem, $|\pi(2^f+1)|\geqslant 3$ which is a contradiction. Thus $p=2$ and $q=2^{2^{n}}$.
If $n=2$, then $q+1=17$ which is impossible.
Thus $n\geqslant 3$, then $255=2^8-1 \mid 2^f-1$ and $|\pi(q-1)|\geqslant 3$ which is again a contradiction.
 Therefore $f=p$ is a prime and the proof is completed.
 \end{proof}

 In the sequel of this section, we let $p$ be a prime, $f\geqslant 1$ be an integer, $q=p^f$ and $S\cong \rm{PSL}_2(q)$.
  When $q$ is an odd prime power and $|\pi(S)|=4$, there is a useful result on possible values for $q$.
 \begin{lemma}\label{oddfour}\cite{[e]}
 If $q$ is odd and $|\pi(q^2-1)|=3$, then either:\\
 \textbf{a)} $q\in \{3^4,5^2,7^2\}$,\\
\textbf{ b)} $p=3$ and $f$ is an odd prime, or\\
 \textbf{c)} $p\geqslant 11$ and $f=1$.
 \end{lemma}

 \begin{lemma}\label{lw}\cite{[non]}
 Let $q\geqslant 5$, $f\geqslant 2$ and $q\neq9$. If $S\leqslant G\leqslant \rm{Aut}(S)$, then $G$ has irreducible characters of degrees $(q+1)[G:G\cap \rm{PGL}_2(q)]$ and $(q-1)[G:G\cap \rm{PGL}_2(q)]$.
 \end{lemma}

Suppose $G$ is a finite group, $q\geqslant 11$
 and $G/R(G)=S$. In the sequel, we consider a given character $\theta \in \rm{Irr}(R(G))$ and try to determine $\rm{cd}(G|\theta)$ in some special cases. For this purpose, we assume that $I:=I_G(\theta)$ and $N:=I/R(G)$.

\begin{lemma}\label{Frobenius}
Suppose $N$ is a Frobenius group whose kernel is an elementary abelian $p$-group. Then one of the following cases holds:\\
\textbf{a)} $\theta$ is extendible to $I$ and $\rm{cd}(G|\theta)=\{\theta(1)[G:I],\theta(1)b\}$, for some positive integer $b$ divisible by $(q^2-1)/(2,q-1)$.\\
\textbf{b)} $\theta$ is not extendible to $I$ and all character degrees in $\rm{cd}(G|\theta)$ are divisible by $p(q+1)\theta(1)$.
\end{lemma}

\begin{proof}
Let $Q/R(G)$ be the $p$-Sylow subgroup of $N$. If $\theta$ extends to $Q$, then $\theta$ extends to $I$ and by Gallagher's Theorem, $\rm{cd}(I|\theta)=\{\theta(1), \theta(1)[I:Q]\}$. Hence using Clifford's Theorem, $\rm{cd}(G|\theta)=\{\theta(1)[G:I], \theta(1)[G:Q]\}$. Therefore as $(q^2-1)/(2,q-1)$ divides $[G:Q]$, $\rm{cd}(G|\theta)=\{\theta(1)[G:I], \theta(1)b\}$, where $b$ is a positive integer divisible by $(q^2-1)/(2,q-1)$. Thus we assume that $\theta$ is not extendible to $Q$. By Lemma \ref{method}, all character degrees in  $\rm{cd}(I|\theta)$ are divisible by $p\theta(1)$. Therefore by Clifford's Theorem, all character degrees in $\rm{cd}(G|\theta)$ are divisible by $\theta(1)p(q+1)$.
\end{proof}

\begin{lemma}\label{special}
Suppose $N\cong \rm{PSL}_2(p^m)$, where $m \neq f$ is a positive divisor of $f$, $p^m\neq 9$ and $p^m\geqslant 7$. Then $\theta(1)(p^m-1)[G:I],\theta(1)(p^m+1)[G:I]\in \rm{cd}(G|\theta) $, where $[G:I]=p^{f-m}(p^{2f}-1)/(p^{2m}-1)$.
\end{lemma}

\begin{proof}
Using Clifford's Theorem and this fact that $\rm{SL}_2(p^m)$ is the Schur representation group of $N$, we have nothing to prove.
\end{proof}

\begin{lemma}\label{general}
Let $N\cong \rm{PGL}_2(p^m)$, where $p$ is odd, $2m$ is a positive divisor of $f$ and  $p^m\neq 3$. Then some $m_0\in \rm{cd}(G|\theta)$ is divisible by $\theta(1)p^{f-m}(p^f+1)$.
\end{lemma}

\begin{proof}
Since $N\cong \rm{PGL}_2(p^m)$, there is a normal subgroup $H$ of $N$ such that $H\cong  \rm{PSL}_2(p^m)$.
Thus for some normal subgroup $J$ of $I$, $J/R(G)=H$. It is easy to see that $\theta$ is $J$-invariant.
Therefore looking at the character table of the Schur representation group $\Gamma$ of $H$, for some $\lambda \in \rm{Irr}(Z(\Gamma))$, $\rm{cd}(J|\theta)=\{\theta(1)m|\,m\in \rm{cd}(\Gamma|\lambda)\}$.
 It is clear that for every $\lambda \in \rm{Irr}(Z(\Gamma))$, there exists $m(\lambda) \in  \rm{cd}(\Gamma|\lambda)$ such that $m(\lambda)$ is even.
 Therefore for some $\lambda_0 \in \rm{Irr}(Z(\Gamma))$, $n:=m(\lambda_0)\theta(1)\in \rm{cd}(J|\theta)$ is even. There is $\varphi \in \rm{Irr}(J|\theta)$ so that $\varphi(1)=m(\lambda_0)\theta(1)$.
 Now let $\chi \in \rm{Irr}(I|\theta)$ be a constituent of $\varphi^I$. Clearly, $\chi(1)$ is divisible by $2\theta(1)$. Hence $m_0:=\chi^G(1)\in \rm{cd}(G|\theta)$ is divisible by $2\theta(1)p^{f-m}(p^f+1)(p^f-1)/2(p^{2m}-1)$.
  Since $f$ is divisible by $2m$, we deduce that $m_0$ is divisible by $\theta(1)p^{f-m}(p^f+1)$.
\end{proof}

Now we state a useful assertion which will be required in the next sections.

\begin{lemma}\label{power}
Let $G$ be a finite group, $p$ be a prime larger than $3$, $q:=2^p$, $S\cong \rm{PSL}_2(q)$ and $G/R(G)=\rm{Aut}(S)$. Also let $\theta \in \rm{Irr}(R(G))$, $I:=I_G(\theta)$ and $N:=I/R(G)$.
 If $S\subseteq N$, then  $p(q-1)\theta(1), p(q+1)\theta(1) \in \rm{cd}(G|\theta)$.
\end{lemma}

\begin{proof}
If $N= S$, then using Clifford's Theorem and this fact that the Schur multiplier of $S$ is trivial, we have nothing to prove. Thus we can assume that $N=\rm{Aut}(S)$. Using Fermat's Lemma, we deduce $p\notin \pi(S)$.  Thus as the Schur multiplier of $N$ is trivial, $\theta$ is extendible to $G$ and by Gallagher's Theorem, $\rm{cd}(G|\theta)=\{m\theta(1)|\,m\in \rm{cd(Aut}(S))\}$. Therefore by lemma \ref{lw}, $p(q-1)\theta(1), p(q+1)\theta(1) \in \rm{cd}(G|\theta)$.
\end{proof}

We end this section with an interesting result. For this purpose we require the following lemmas.

\begin{lemma}\label{triangle}
Let $G$ be a finite non-abelian group such that $\Delta(G)^c$ is not bipartite. Then there exists a normal subgroup $N$ of $G/R(G)$ in which for some prime $u$ and positive integer $\alpha \geqslant 1$, $N\cong \rm{PSL}_2(u^\alpha)$. Also there exist $u^\prime \in \pi (u^\alpha-1)-\{2\}$ and $u^{''}\in \pi (u^\alpha+1)-\{2\}$ such that $\Delta(G)^c[\{u,u^\prime, u^{''}\}] $ is a triangle.
\end{lemma}

\begin{proof}
Since  $\Delta(G)^c$ is not bipartite, there exists $\pi \subseteq \rho(G)$ such that $\pi$ is the set of vertices of an  odd cycle with minimal length of $\Delta(G)^c$. Using Lemma \ref{cycle}, $L:=O^{\pi^\prime}(G)=S\times A$, where $A$ is abelian, $S\cong \rm{SL}_2(u^\alpha)$ or $S\cong \rm{PSL}_2(u^\alpha)$ for a prime $u\in \pi$ and a positive integer $\alpha$, and the primes in $\pi - \{u\}$ are alternately odd divisors of $u^\alpha+1$ and  $u^\alpha-1$. Thus $N:=\rm{LR}(G)/R(G)\cong \rm{PSL}_2(u^\alpha)$ is a normal subgroup of $G/R(G)$. We claim that in $\Delta(G)^c$, $u$ is adjacent to all vertices in $\pi-\{u\}$. On the contrary, we assume that there exists $x\in \pi-\{u\}$ such that $x$ and $u$ are adjacent vertices in $\Delta(G)$. Then for some $\chi \in \rm{Irr}(G)$, $xu|\chi(1)$. Now let $\theta\in \rm{Irr}(L)$ be a constituent of $\chi_L$. Using Lemma \ref{fraction}, $\chi(1)/\theta(1)$ divides $[G:L]$. Thus as $G/L$ is a $\pi^\prime$-group, $xu|\theta(1)$. It is a contradiction as $x$ and $u$ are non-adjacent vertices in  $\Delta(L)$. Thus $|\pi|=3$ and we are done.
\end{proof}

\begin{lemma}\label{ote}
Let $G$ be a finite group, $R(G)< M\leqslant G$, $S:=M/R(G)$ be isomorphic to $ \rm{PSL}_2(q)$, where for some prime $p$ and positive integer $f\geqslant 1$, $q=p^f$, $|\pi(S)|\geqslant 4$ and $S\leqslant G/R(G)\leqslant \rm{Aut}(S)$. Also let $\theta\in \rm{Irr}(R(G))$, $I:=I_M(\theta)$ and $N:=I/R(G)$. If $\Delta(G)^c$ is not  bipartite,  then $N=S$.
\end{lemma}

\begin{proof}
Since $\Delta(G)^c$ is not  bipartite, by Lemma \ref{triangle}, there exists a normal subgroup $L$ of $G/R(G)$ in which for some prime $u$ and positive integer $\alpha$, $L\cong \rm{PSL}_2(u^\alpha)$. Also there exist $u^\prime \in \pi (u^\alpha-1)-\{2\}$ and $u^{''}\in \pi (u^\alpha+1)-\{2\}$ such that the induced subgraph of  $\Delta(G)^c$ on $\pi:=\{u,u^\prime, u^{''}\} $ is a triangle. Thus as $L$ is a non-abelian minimal normal subgroup of $G/R(G)$, $q=u^\alpha$.
 On the contrary, we assume that $N\neq S$. Then using Dickson's list, one of the following cases holds. Note that $d:=(2,q-1)$.\\
Case 1.  $N$ is an elementary abelian $p$-group. Then by Clifford's Theorem, there is $m\in \rm{cd}(M|\theta)$ such that $m$ is divisible by $\theta(1)(q^2-1)/d$. It is a contradiction as $\Delta(G)^c[\pi]$ is a triangle.\\
Case 2. $N$ is contained in a dihedral group. Using Clifford's Theorem, for some $m\in \rm{cd}(M|\theta)$ and $\epsilon \in \{\pm 1\}$,   $m$ is divisible by $\theta(1)p(q+\epsilon)/d$. It is a contradiction as $\Delta(G)^c[\pi]$ is a triangle. \\
Case 3. $N$ is a Frobenius group whose kernel is an elementary abelian $p$-group. Then using Lemma \ref{Frobenius},  there is  $m\in \rm{cd}(M|\theta)$ so that $m$ is divisible by $\theta(1)(q^2-1)/d$ or $\theta(1)p(q+1)$ and we again obtain a contradiction.\\
Case 4.  $N\cong A_5$. Then as $\rm{SL}_2(5)$ is the Schur representation group of $A_5$, there exists $m\in \rm{cd}(M|\theta)$ such that $m$ is divisible by  $q(q^2-1)\theta(1)/20d$ which is impossible as  $\Delta(G)^c[\pi]$ is a triangle.\\
Case 5. $N\cong A_4$. Then using Clifford's Theorem, all character degrees in  $ \rm{cd}(M|\theta)$ are divisible by $\theta(1)q(q^2- 1)/12d$. It is again a contradiction.\\
Case 6. $N\cong S_4$ and $16|q^2-1$. Then using Clifford's Theorem, all character degrees in  $ \rm{cd}(M|\theta)$ are divisible by $\theta(1)q(q^2- 1)/24d$  which is impossible.\\
Case 7. $N \cong \rm{PSL}_2(p^m)$, where $m\neq f$ is a positive divisor of $f$, $p^m\geqslant 7$ and $p^m\neq 9$. By Lemma \ref{special}, $b:=\theta(1)p^{f-m}(p^m+1)(p^f+1)(p^{f}-1)/(p^{2m}-1) \in \rm{cd}(M|\theta)$. Thus as $p(q+1)$ divides $b$, we again obtain a contradiction.\\
 Case 8. $N\cong \rm{PSL}_2(9)$, $p=3$ and $f\geqslant 4$ is even. Then using Clifford's Theorem, all character degrees in  $ \rm{cd}(M|\theta)$ are divisible by $3(q^2- 1)/80$.  It is a contradiction as $\Delta(G)^c[\pi]$ is a triangle.\\
 Case 9. $N\cong \rm{PGL}_2(p^m)$, where $2m$ is a positive divisor of $f$ and $p^m\neq 3$.  By Lemma \ref{general}, some $m_0\in \rm{cd}(M|\theta)$ is divisible by $\theta(1)p^{f-m}(p^{f}+1)$  which is again a contradiction.
\end{proof}

\begin{lemma}\label{direct product}
Assume that $f\geqslant 2$ is an integer, $q=2^f$,  $S\cong \rm{PSL}_2(q)$ and $G$ is a finite group such that $G/R(G)=S$. If $\Delta(G)[\pi(S)]=\Delta(S)$, then $G\cong S\times R(G)$.
\end{lemma}

\begin{proof}
 Let $H$ be the last term of the derived series of $G$. Since $G$ is non-solvable, $H$ is non-trivial. Let $N:=H\cap R(G)$. As $H/N \cong HR(G)/R(G)\lhd G/R(G)\cong  S$, we deduce that $H/N\cong S$. We claim $N=1$. On the contrary, we assume that $N\neq 1$. Thus there exists a non-principal linear character $\lambda \in \rm{Irr}(N)$. Now let $f=2$ or $3$. Hence $S\cong A_5$ or $\rm{PSL}_2(8)$ and $\Delta(S)$ is an empty graph. Using Lemma \ref{good}, for some $\chi \in \rm{Irr}(H|\lambda)$, $\chi(1)$ is divisible by two distinct primes in $\pi(S)$. Therefore $\Delta(G)[\pi(S)]$ is a non-empty subgraph of $\Delta(G)$. It is a contradiction. Therefore $f\geqslant 4$.
  Since $\Delta(G)[\pi(S)]=\Delta(S)$, $\Delta(H)\subseteq \Delta(G)$ is not a bipartite graph. Hence by Lemma \ref{ote} and this fact that the Schur multiplier of $S$ is trivial, $\lambda$  extends to $H$. It is a contradiction as $H$ is perfect. Hence $N=1$ and $G\cong H\times R(G) \cong S\times R(G)$.
\end{proof}

\section{The proof of Main Theorem}
$\noindent$ In this section, we wish to prove our main result.

\begin{lemma}\label{free}
Suppose $G$ is a finite group and $\Delta(G)$ is a $K_4$-free graph withe seven vertices. Then there exists a normal subgroup $R(G)<M\leqslant G$ so that $G/R(G)$ is an almost simple group with socle $S:=M/R(G)\cong \rm{PSL}_2(q)$, where $q$ is a prime power and $\Delta(G)^c[\pi(S)]$ is not a bipartite graph. Also if $\pi\subseteq \rho(G)$ is the set of vertices of an odd cycle in $\Delta(G)^c$, then $\pi\subseteq \pi (S)$.
\end{lemma}

\begin{proof}
As $\Delta(G)$ is a $K_4$-free graph with seven vertices, $\Delta(G)^c$ is not bipartite. Now assume that $\pi\subseteq \rho(G)$ is the set of vertices of an odd cycle in $\Delta(G)^c$. By Lemma \ref{cycle}, $N:=O^{\pi^\prime}(G)=R\times A$, where $A$ is abelian, $R\cong \rm{SL}_2(q)$ or $R\cong \rm{PSL}_2(q)$ for a prime power $q$ and $\pi \subseteq \pi(\rm{PSL}_2(q))$. Let $M:=NR(G)$. Then $S:=M/R(G)\cong N/R(N)\cong \rm{PSL}_2(q)$ is a non-abelian minimal normal subgroup of $G/R(G)$. Note that $\pi\subseteq \pi(S)$ and $\Delta(G)^c[\pi(S)]$ is not a bipartite graph.

Let $C/R(G)=C_{G/R(G)}(M/R(G))$. We claim that $C=R(G)$ and $G/R(G)$ is an almost simple group with socle $S=M/R(G)$. Suppose on the contrary that $C\neq R(G)$ and let $L/R(G)$ be a chief factor of $G$ with $L\leqslant C$. Then $L/R(G)\cong T^k$, for some non-abelian simple group $T$ and some integer $k\geqslant 1$.
 As $L\leqslant C$, $LM/R(G)\cong L/R(G)\times M/R(G) \cong S\times T^k$. Let $F:=\pi (S)\cap \pi(T)$. Since $\Delta(G)$ is $K_4$-free, $|F|\leqslant 3$. Note that $|\rho(S)|,|\rho(T)|\geqslant 3$ and $2\in F$. Now one of the following cases occurs:\\
 Case 1. $|F|=1$. Then $F=\{2\}$. Hence $3$ does not divide $|T|$. Therefore by \cite{suzuki},  $T$ is a Suzuki simple group. By Lemma \ref{her}, $|\pi(T)|\geqslant 4$. Thus $\Delta(S\times T^k)$ contains a copy of $K_4$ and it is a contradiction.\\
Case 2. $|F|=2$. Then as $|\pi(S)|,|\pi(T)|\geqslant 3$, we again deduce that $\Delta(S\times T^k)$ contains a copy of $K_4$ and it is a contradiction.\\
Case 3. $|F|= 3$. If $|\pi(S)|$ or $|\pi(T)|$ is larger than $3$, then $\Delta(S\times T^k)$ contains a copy of $K_4$ which is a contradiction. Thus $\rho(S)=\rho(T)=F$. Hence $\Delta(S\times T^k)$ is a triangle. It is a contradiction as $\Delta(G)^c[\pi(S)]$ is not a bipartite graph. It completes the proof.
\end{proof}

 Let $G$ be a finite group and $\Delta(G)$ be a $K_4$-free graph with seven vertices. By Lemma \ref{free}, there exists a normal subgroup $R(G)<M\leqslant G$ so that $G/R(G)$ is an almost simple group with socle $S:=M/R(G)\cong \rm{PSL}_2(q)$ where $q$ is a prime power and $\Delta(G)^c[\pi(S)]$ is not a bipartite graph. Also if $\pi\subseteq \rho(G)$ is the set of vertices of an odd cycle in $\Delta(G)^c$, then $\pi\subseteq \pi (S)$.
 Since $\Delta(G)$ is $K_4$-free, $\Delta(S)$ is too. The structure of $\Delta(S)$ is determined by  Lemma \ref{chpsl}. Note that as $\Delta(G)^c[\pi(S)]$ is not a bipartite graph, $\Delta(S)\ncong K_2\cup K_1, K_3\cup K_1$. In the sequel of this paper, we use the notation $\pi_R$ for the set $\rho(R(G))-\pi(G/R(G))$. 

\begin{lemma}\label{easy}
Let $|\pi_R|\geqslant 3$ and $|\pi(S)|\geqslant 4$. Then $G$ does not exist.
\end{lemma}

\begin{proof}
By Lemma \ref{pal}, there are $p,q \in \pi_R$ so that $p$ and $q$ are adjacent vertices in $\Delta(R(G))\subseteq \Delta(G)$. There is $\theta \in \rm{Irr}(R(G))$ with $pq|\theta(1)$. Therefore using Lemma \ref{good}, for some $\chi \in \rm{Irr}(M|\theta)$, $\chi(1)/ \theta(1)$ is divisible by at least two distinct prime divisors of $|S|$. Thus $|\pi(\chi(1))|\geqslant 4$ which is impossible.
\end{proof}

\begin{lemma}\label{threevertex}
If $\Delta(S)$ has three vertices, then $G\cong S\times A\times B$, where\\
 \textbf{a)} $S\cong A_5$ or $\rm{PSL}_2(8)$,\\
 \textbf{b)} $A$ and $B$ are disconnected groups of disconnected Types 1 or 4. Also $|\rho(A)|=|\rho(B)|=2$, and\\
\textbf{c)} $\Delta(G)\cong C_4\star K_3^c$.
\end{lemma}

\begin{proof}
Since $\Delta(G)^c[\pi(S)]$ is not a bipartite graph, $\Delta(G)[\pi(S)]$ is an empty graph. Thus using Lemma \ref{her}, we can see that $G=M$ and $S\cong A_5$ or $\rm{PSL}_2(8)$. Hence by Lemma \ref{direct product},  $G\cong  S\times R(G)$. Thus as $\Delta(G)[\pi(S)]$ is an empty graph, $\pi(S)\cap \rho(R(G))=\emptyset$ and $\rho(R(G))=\pi_R$. Also as $\Delta(G)$ is $K_4$-free, $\Delta(R(G))$ is triangle free. Hence as $|\rho(G)|=7$, using Lemma \ref{square}, $\Delta(R(G))\cong C_4$. Thus using Lemma \ref{square2}, $R(G)\cong A\times B$, where $A$ and $B$ are disconnected groups with $|\rho(A)|=|\rho(B)|=2$. Therefor by Lemma \ref{type},  $A$ and $B$ are of disconnected Types 1 or 4. It is easy to see that $\Delta(G)\cong C_4\star K_3^c$ and the proof is completed.
\end{proof}

\begin{lemma}\label{one7}
 $\Delta(S)\ncong K_1\cup K_1 \cup K_2$.
\end{lemma}

\begin{proof}
By Lemmas \ref{chpsl} and \ref{interest}, either $S\cong \rm{PSL}_2(16)$ or $ \rm{PSL}_2(2^p)$, where $p\geqslant 5$ is a prime, $t:=2^p-1$ is a Mersenne prime and $|\pi(2^p+1)|=2$. Now one of the following cases occurs: \\
Case 1. $S\cong \rm{PSL}_2(16)$ or $G=M$. It is easy to see that $\pi(G/R(G))=\pi(S)$. Thus as $|\pi(S)|=4$ and $|\rho(G)|=7$, $|\pi_R|=3$. Hence by Lemma \ref{easy}, we have nothing to prove.\\
Case 2. $S\cong \rm{PSL}_2(2^P)$ and $G\neq M$. Clearly, $G/R(G)=\rm{Aut}(S)$. Thus as $|\pi(\rm{Aut}(S))|=5$ and $|\rho(G)|=7$, $\pi_R\neq \{\emptyset\}$. Let $b\in \pi_R$. There exists $\theta_b \in \rm{Irr}(R(G))$ such that $b|\theta_b(1)$. By Lemma \ref{ote}, $\theta_b$ is $M$-invariant. Thus by Lemma \ref{power}, $m:=\theta_b(1)p(2^p+1)\in \rm{cd}(G|\theta_b)$ which is a contradiction as $|\pi(m)|\geqslant 4$.
\end{proof}

\begin{lemma}\label{two7}
If $\Delta(S)\cong K_2\cup K_1 \cup K_2$, then $G\cong S\times R(G)$, where\\
\textbf{a} $S\cong \rm{PSL}_2(2^6)$, $\rm{PSL}_2(2^9)$ or $\rm{PSL}_2(2^p)$, where $p$ is prime and $|\pi(2^p\pm 1)|=2$,\\
\textbf{b)} $R(G)$ is a disconnected  group of disconnected Type 1 or 4, and\\
\textbf{c)} $\Delta(G) \cong (K_2\cup K_1\cup K_2)\star K_2^c$.
\end{lemma}

\begin{proof}
By Lemmas \ref{chpsl} and \ref{evenfive}, $S$ is isomorphic to $\rm{PSL}_2(2^f)$, where  $f$ is either $6,9$ or a prime $p$ with $|\pi(2^p-1)|=|\pi(2^p+1)|=2$. Thus we can see that $|\pi(G/R(G))|=5$ or $6$. Hence as $|\rho(G)|=7$, $\pi_R$ is non-empty. Suppose  $q\in \pi_R$ and $\theta \in \rm{Irr}(R(G))$ with $q|\theta(1)$.  Using Lemma \ref{ote}, $\theta$ is $M$-invariant. Then as the Schur multiplier of $S$ is trivial, by Gallagher's Theorem, $q$ is adjacent to all vertices in $\pi(S)$.
Now let $G\neq M$. Thus $[G:M]$ is divisible by either $2,\,3$ or $p$.  By Lemma \ref{lw}, $[G:M](2^f-1)$, $[G:M](2^f+1)\in \rm{cd}(G)$. Note that $\pi_R\neq \emptyset$. Let $t\in \pi_R$. If $f=6$ or $9$, then the induced subgraph of $\Delta(G)$ on $\pi(S)\cup \{t\}$ contains a copy of $K_4$ which is impossible. Thus we assume that $f=p$. Then using Lemma \ref{power} and this fact that every $\theta  \in \rm{Irr}(R(G))$ with $t|\theta(1)$, is $M$-invariant, the induced subgraph of $\Delta(G)$ on $X:=\{p,t\} \cup \pi(2^p-1)$ is a copy of $K_4$ and  we have a contradiction. Therefore $G=M$. Since $\pi(G/R(G))=\pi(S)$, $|\rho(G)|=7$ and $|\pi(S)|=5$, there exist primes $q$ and $q^\prime$ such that $\pi_R=\{q,q^\prime\}$. Note that $q$ and $q^\prime$ are adjacent to all vertices in $\pi(S)$.  Now we claim that $\Delta(G)[\pi(S)]=\Delta(S)$. On the contrary, we assume that there are $x,y\in \pi(S)$ such that $x$ and $y$ are joined by an edge in $ E(\Delta(G))-E(\Delta(S))$. Then for some $\chi \in \rm{Irr}(G)$, $xy$ divides $\chi(1)$. Let $\tilde{\theta} \in \rm{Irr}(R(G))$ be a constituent of $\chi_{R(G)}$. Then using  Lemma \ref{ote},
 $\tilde{\theta}$ is $G$-invariant. Then as the Schur multiplier of $S$ is trivial, by Gallagher's Theorem, $\rm{cd}(G|\tilde{\theta})=\{m\tilde{\theta}(1)|\,m\in \rm{cd}(S)\}$. Hence as $\chi(1)\in \rm{cd}(G|\tilde{\theta})$, $\tilde{\theta}(1)$ is divisible by $x$ or $y$. Without loss of generality, we assume that $x$ divides $\tilde{\theta}(1)$. There is $\epsilon\in \{\pm 1\}$ so that $x\notin \pi(2^f+\epsilon)$. Thus the induced subgraph of $\Delta(G)$ on $X:=\pi(2^f+\epsilon)\cup\{x,q\}$ is a copy of $K_4$ and it is a contradiction.
Therefore $\Delta(G)[\pi(S)]=\Delta(S)$. Hence by Lemma \ref{direct product}, $G \cong S\times R(G)$. If $\rho (R(G))\cap \pi(S)\neq \emptyset$, then the induced subgraph of $\Delta(G)$ on $\pi(S)\cup \{q\}$ contains  a copy of $K_4$ and it is a contradiction. Thus as $\Delta(G)$ is $K_4$-free, it is easy to see that $\rho(R(G))=\pi_R=\{q,q^\prime\}$ and $\Delta(R(G))\cong K_2^c$. Therefore by Lemma \ref{type}, $R(G)$ is a disconnected group of disconnected Type 1 or 4. Finally, it is easy to see that, $\Delta(G)\cong \Delta(S) \star K_2^c\cong (K_2\cup K_1 \cup K_2)\star K_2^c$ and the proof is completed.
\end{proof}

\begin{lemma}\label{three7}
 $\Delta(S)\ncong K_3\cup K_1 \cup K_1$.
\end{lemma}

\begin{proof}
By Lemma \ref{chpsl}, for some $\epsilon \in \{\pm 1\}$, there exists a prime $t$ such that $\pi(q-\epsilon)=\{t\}$ and $|\pi(q+\epsilon)|=3$. Let $x\in \pi(G/R(G))-\pi(S)$. Then by Lemma \ref{lw}, for some $m\in \rm{cd}(G/R(G))\subseteq \rm{cd}(G)$, $m$ is divisible by $x(q+\epsilon)$. It is a contradiction as $|\pi(m)|\geqslant 4$.   Hence $\pi(G/R(G))=\pi(S)$. Therefore as $|\pi(S)|=5$ and $|\rho(G)|=7$, there exist distinct primes $x$ and $y$ such that $\pi_R=\{x,y\}$. Suppose $x$ and $y$ are adjacent vertices in $\Delta(G)$. Then there exists  $\theta\in \rm{Irr}(R(G))$ such that $xy|\theta(1)$. Thus using Lemma \ref{good}, for some $\chi \in \rm{Irr}(M|\theta)$, $|\pi(\chi(1))|\geqslant 4$ and
 we obtain a contradiction. Hence $x$ and $y$ are non-adjacent vertices in $\Delta(G)$. We claim that no prime in $\pi_R$ is adjacent to any prime in $\pi(q+\epsilon)$. On the contrary, we assume that there exist $v\in \pi_R$ and $w\in \pi(q+\epsilon)$ such that $v$ and $w$ are adjacent vertices in $\Delta(G)$. Thus for some $\chi \in \rm{Irr}(G)$, $vw|\chi(1)$. Now let $\varphi \in \rm{Irr}(M)$ and  $\theta \in \rm{Irr}(R(G))$ be  constituents of $\chi_{M}$ and $\varphi_{R(G)}$, respectively. By Lemma \ref{fraction}, $v|\theta(1)$. Using Lemma \ref{ote},
 $\theta$ is $M$-invariant. As the Schur multiplier of $S$ is trivial, by Gallaghers Theorem, $(q+\epsilon)\theta(1)\in \rm{cd}(M|\theta)$. Which is impossible as $|\pi(\theta(1)(q+\epsilon))|\geqslant 4$ .
Therefore  no prime in $\pi_R$ is adjacent to any prime in $\pi(q+\epsilon)$. Thus for every $z\in \pi(q+\epsilon)$, $\pi:=\{x,y,z\}$ is the set of vertices of an odd cycle in $\Delta(G)^c$. Hence using Lemma \ref{free}, $\pi\subseteq \pi (S)$ which is a contradiction as $\pi_R\cap \pi (S)=\emptyset$.
\end{proof}

\begin{lemma}\label{four7}
 $\Delta(S)\ncong K_2\cup K_1\cup K_3$.
\end{lemma}

\begin{proof}
Using Lemma \ref{chpsl}, there exists $\epsilon \in \{\pm 1\}$ such that  $|\pi(q-\epsilon)|=2$ and $|\pi(q+\epsilon)|=3$. Let $x\in \pi(G/R(G))-\pi(S)$.  Then by Lemma \ref{lw}, for some $m\in \rm{cd}(G/R(G))\subseteq \rm{cd}(G)$, $m$ is divisible by $x(q+\epsilon)$. It is a contradiction as $|\pi(m)|\geqslant 4$. Hence $\pi(G/R(G))=\pi(S)$. Therefore as $|\pi(S)|=6$ and $|\rho(G)|=7$, there exists a prime $x$ such that $\pi_R=\{x\}$. Thus for some $\theta \in \rm{Irr}(R(G))$, $x|\theta(1)$.  Hence using Lemma \ref{ote}, $\theta$ is $M$-invariant. Thus as the Schur multiplier of $S$ is trivial,  Gallagher's Theorem implies that  $m:=\theta(1)(q+\epsilon)\in \rm{cd}(M|\theta)$ which is a contradiction as $|\pi(m)| \geqslant 4$.
\end{proof}

\begin{lemma}\label{five7}
If $\Delta(S)\cong K_3\cup K_1\cup K_3$, then $G\cong \rm{PSL}_2(q)\times R(G)$, where $R(G)$ is abelian, for some integer $f\geqslant 1$, $q=2^f$ and $|\pi(q-1)|=|\pi(q+1)|=3$. Also $\Delta(G) \cong K_3\cup K_1\cup K_3$.
\end{lemma}

\begin{proof}
By Lemma \ref{chpsl}, $S\cong \rm{PSL}_2(q)$, where for some integer $f\geqslant 1$, $q=2^f$ and $|\pi(q-1)|=|\pi(q+1)|=3$. We first show that $G=M$. On the contrary, we assume that $G \neq M$. Then there is a prime $r$ so that $r$ divides $[G:M]=[G/R(G):S]$. Since $q-1$ and $q+1$ are co-prime, there is $\epsilon \in \{\pm 1\}$ with $r$ does not divide $q+\epsilon$. Using Lemma \ref{lw}, for some $m\in \rm{cd}(G)$, $m$ is divisible by $r(q+\epsilon)$.  Hence $|\pi(r(q+\epsilon))|=4$ and $\Delta(G)[\pi(m)]$ has a copy of $K_4$ which is impossible. Thus $G=M$ and $G/R(G)=S$. We claim that $\Delta(G)[\pi(S)]=\Delta(S)$. On the contrary, we assume that there exist non-adjacent vertices $x$ and $y$ of $\Delta(S)$ such that $x$ and $y$ are adjacent vertices in $\Delta(G)$. Thus for some $\chi\in \rm{Irr}(G)$, $\chi(1)$ is divisible by $xy$. Let $\theta$ be a constituent of $\chi_{R(G)}$. Using Lemma \ref{ote}, $\theta$ is $G$-invariant. Thus as the Schur multiplier of $S$ is trivial, $\theta$ extends to $G$. By Gallagher's Theorem, $\rm{cd}(G|\theta)=\{m\theta(1)|\;m\in \rm{cd}(S)\}$.  Thus as $\chi(1)\in \rm{cd}(G|\theta)$ is divisible by $xy$, $\theta(1)$ is divisible by $x$ or $y$. Hence one of the character degrees $\theta(1)(q-1)$ or $\theta(1)(q+1)$ of $G$ is divisible by four distinct primes which is a contradiction.
 Therefore $\Delta(G)[\pi(S)]=\Delta(S)$. Thus using Lemma \ref{direct product}, $G\cong S\times R(G)$.
 Now let $p\in \rho(R(G))$. There is $\epsilon \in \{\pm 1\}$ such that $p\notin \pi(q+\epsilon)$. Thus $\Delta(G)[\pi(p(q+\epsilon))]$ is a copy of $K_4$ which is a contradiction. Hence $R(G)$ is abelian and we are done.
\end{proof}

\begin{lemma}\label{six7}
 $\Delta(S)\ncong ((K_1\cup K_1)\star K_1)\cup K_1$.
\end{lemma}

\begin{proof}
Using Lemma \ref{oddfour} and Fermat's Lemma, $\pi(G/R(G))=\pi(S)$, unless $S\cong \rm{PSL}_2(3^f)$, where $f\geqslant 5$ is  an odd prime and $f|[G/R(G):S]$. We first assume that $\pi(G/R(G))=\pi(S)$. Then as $|\pi(S)|=4$ and $|\rho(G)|=7$, $|\pi_R|=3$.  Hence by Lemma \ref{easy}, we are done. Now let $S\cong \rm{PSL}_2(3^f)$ where $f\geqslant 5$ is an odd prime and $f|[G/R(G):S]$. Since $[G:M]=f$ or $2f$, $\rho(G/R(G))=\pi(S) \cup \{f\}$. Hence as $|\rho(G)|=7$, for some distinct primes $p$ and $q$, $\pi_R=\{p,q\}$. For every $b\in \pi_R$, let $\theta_b\in \rm{Irr}(R(G))$ be a
     character so that $b$ divides $\theta_b(1)$. By Lemma \ref{ote}, $\theta_p$ and $\theta_q$ are $M$-invariant. Thus using Lemma \ref{lw} and this fact that $\rm{SL}_2(3^f)$ is the Schur representation group of $S$, every prime in $\pi:=\{p,q,f\}$ is adjacent to all vertices in $\pi(3^{2f}-1)$. Hence as $\Delta(G)$ is $K_4$-free, $\Delta(G)^c[\pi]$ is a triangle. Therefore by Lemma \ref{free}, $\pi\subseteq \pi(S)$ which is impossible as $\pi_R \cap \pi(S)=\emptyset$.
\end{proof}

\begin{lemma}\label{oddPSL7}
$\Delta(S)\ncong ((K_1\cup K_2)\star K_1)\cup K_1$ or $((K_2\cup K_2)\star K_1)\cup K_1$.
\end{lemma}

\begin{proof}
By Lemma \ref{chpsl}, $S\cong \rm{PSL}_2(q)$, where $q$ is an odd prime power. If for some prime $p$, $p\in \pi(G/R(G))-\pi(S)$, then as $\Delta(S)$ has a triangle, using Lemma \ref{lw}, we can see that $\Delta(G/R(G))\subseteq \Delta(G)$ has a copy of $K_4$ and it is a contradiction. Thus $\pi(G/R(G))=\pi(S)$ and $\pi_R\neq \emptyset$.  Hence there is a prime $p$ with $p\in \pi_R$. There is $\theta \in \rm{Irr}(R(G))$ so that $p$ divides $\theta(1)$.  By Lemma \ref{ote}, $\theta$ is $M$-invariant. Thus as $\rm{SL}_2(q)$ is the Schur representation group of $S$, $\theta(1)(q-1)$, $\theta(1)(q+1)\in \rm{cd}(M|\theta)$. There exists $\epsilon \in \{\pm 1\}$ with $|\pi(q+\epsilon)|=3$. Hence $\Delta(M)[\pi(\theta(1)(q+\epsilon))]\subseteq \Delta(G)$ contains a copy of $K_4$ and it is a contradiction.
\end{proof}

\section*{Acknowledgements}
This research was supported in part
by a grant  from School of Mathematics, Institute for Research in Fundamental Sciences (IPM).
 I would like to express
my gratitude to the referee for valuable comments which
improved the original version.
I also would like to thank Sadegh Nazardonyavi for his help during the preparation of the manuscript.


\end{document}